\documentclass{amsart}

\usepackage{graphicx}
\usepackage{mathtools}
\mathtoolsset{showonlyrefs}
\usepackage{enumitem}
\usepackage{verbatim}

\usepackage{hyperref}

\def \R{\mathbb R}
\def \B{\mathbb B}
\def \N{\mathbb N}
\def \Z{\mathbb Z}
\def \e{\varepsilon}
\newcommand \vol[2][3]{\left|#2\right|_{#1}}
\newcommand \dist {\operatorname{dist}}
\newtheorem{theorem}{Theorem}
\newtheorem{lemma}[theorem]{Lemma}
\newtheorem{definition}[theorem]{Definition}
\newtheorem{proposition}[theorem]{Proposition}
\newtheorem{question}[theorem]{Question}

\title{On convex bodies with constant non-central sections}

\author{J. Haddad, D. Ryabogin}

\begin{document}

\begin{abstract}
	We prove that if $C$ is a symmetric convex body of revolution in $\R^4$ containing the unit Euclidean ball $\B_4$, such that the sections of $C$ by hyperplanes tangent to $\B_4$ have constant area $A>0$, then $C$ is a Euclidean ball, provided $\frac 1{\pi} \arctan((\frac{3A}{4\pi})^{1/3})$ satisfies certain arithmetic properties that can be read from its expansion as a continued fraction.
	We show that the set of values $A$ satisfying these properties has positive Hausdorff dimension.
\end{abstract}

\maketitle

\section{Introduction and main theorem}

One of the central problems in geometric tomography asks to what extent a convex body (a compact convex set with non-empty interior) is determined by the volumes of its sections or projections. A well-known question of this type was formulated by Barker and Larman in \cite{barker2001}. They asked the following:

Let $C_1$ and $C_2$ be convex bodies in $\R^n$ containing the unit Euclidean ball $\B_n$ in their interiors. Assume that for every hyperplane $H$ tangent to $\B_n$,
\begin{equation}
	\label{eq_equalsections}
	\vol[n-1]{C_1 \cap H} = \vol[n-1]{C_2 \cap H},
\end{equation}
where $\vol[n-1]{\ \cdot \ }$ denotes the $(n-1)$-dimensional Lebesgue measure in $H$.
Is it true that $C_1=C_2$?

Despite its simple formulation this problem remains open in general.
The conjecture is related to several classical problems in convex geometry and geometric tomography.

If central sections (sections by hyperplanes containing the origin) are considered, then the answer is positive for origin-symmetric $C_1, C_2$, thanks to the injectivity of the Radon transform.
The available toolbox of central sections is much more developed than for non-central sections, and several techniques from harmonic analysis are available in this case.
For an overview of the relevant problems for central sections see \cite{koldobsky2005fourier, gardner2006geometric}.

Some partial results related to the Barker--Larman conjecture are known.
The case where $n=2$ and $C_2$ is a centered Euclidean ball in the plane had been proved in 1951 by Santaló in \cite{santalo1951two}, where he also proved an analogous result for spherical convex bodies inside the sphere $S^2$.

Barker and Larman in \cite{barker2001} established several variants of the conjecture. For example, they proved that in dimension $2$, the bodies are equal if one assumes that for some $\varepsilon>0$ small, \eqref{eq_equalsections} holds for all lines $H$ at a distance between $1-\varepsilon$ and $1$ from the origin.
A similar result was also proved in odd dimensions.
For sections by subspaces $H$ of codimension larger than $1$, the problem was solved in the affirmative.

Yaskin \cite{yaskin2011unique} proved the conjecture in the case where both $C_1, C_2$ are polytopes in $\R^n$.
Later Yaskin and Yaskina in \cite{yaskin2015thick} proved it in dimension $3$ in the case where $C_2 = r \B_3$, under the additional assumption that the two bodies $C_1$ and $C_2$ have the same volume.

Similar results of determination of convex bodies by non-central sections were established in \cite{alfonseca2025characterizations, alfonseca2025croft, gardner2012problem, kurusa2015characterizations, nazarov2014asymmetric, yaskin2017non}.

A natural special case to consider is when one body is a Euclidean ball and the other one satisfies additional symmetry conditions.
In particular, bodies of revolution allow one to define the convex body in terms of a single function describing its profile.

Our main result gives a positive answer to this question for bodies of revolution in dimension $4$, under an additional condition on the value of the volume of the section, which is arithmetic in nature.

\begin{theorem}
	\label{thm_main}
	There exists an infinite set $\mathbb A \subseteq \R$ of positive Hausdorff dimension with the following property:

	Let $C \subseteq \R^4$ be a symmetric convex body of revolution containing the centered unit Euclidean ball $\B_4$ in its interior, such that for every hyperplane $H$ tangent to $\B_4$, $\vol[3]{C \cap H} = A$ where $A$ is a constant independent of $H$.
	If $A \in \mathbb A$ then $C$ is a Euclidean ball.
\end{theorem}

The proof of Theorem \ref{thm_main} will be split in two parts.
First we prove Theorem \ref{thm_infinite_rotation} which will be stated in Section \ref{sec_rotation}, where an explicit definition of the set $\mathbb A$ will be given.
Later in Section \ref{sec_finite_inequalities} we explain how to verify if a given value $A$ belongs to $\mathbb A$, and in Section \ref{sec_examples} we find infinitely many values $A \in \mathbb A$ and show that $\mathbb A$ has positive Hausdorff dimension.

The symmetry assumption is used in the proof in a non-essential way, and the Theorem probably holds also in the non-symmetric case.
But the symmetry simplifies our proof considerably.

From the proof of Theorem \ref{thm_main}, the examples given in Section \ref{sec_examples}, and some properties of sets with restricted continued fraction, it seems likely that $\mathbb A$ has positive Lebesgue measure.
We did not pursue this direction further.

In Section \ref{sec_area_equation} we study the equation describing the value of $\vol[3]{C \cap H}$ in terms of the slope of $H$ with respect to the axis of revolution.

In Section \ref{sec_rotation} we transform the geometrical problem into a problem about irrational rotations in the circle. Roughly speaking, we will define a sequence of points $u_k \in S^1$ rotating by the constant angle $\arctan((\frac{3A}{4\pi})^{1/3})$, and a sequence of intervals $I_k \subseteq S^1$ which shrink to a point as $k \to \infty$.
We prove that if $u_k \not\in I_k$ for every $k$, then the conclusion of Theorem \ref{thm_main} holds.

In Section \ref{sec_finite_inequalities} we show that this condition can be verified using arithmetic properties of the ``rotation number'' $\gamma = \frac 1\pi \arctan((\frac{3A}{4\pi})^{1/3})$, namely, we establish a finite number of inequalities involving the coefficients of its continued fraction (Theorem \ref{thm_estimates}), and show that these inequalities are satisfied for a subset of $\mathbb A$ of positive Hausdorff dimension.

Lastly, we think it is pertinent to mention an interesting result by Ungar \cite{ungar1954freak} where just as in Theorem \ref{thm_main}, a uniqueness result depends on whether a certain parameter belongs to a complicated set in the real line.
He considered the following question:
Let $f:S^2 \to \R$ be a continuous function such that if $D$ is any spherical cap of radius $\alpha$,
\begin{equation} \label{eq_int0} \int_D f = 0.\end{equation}
Does it follow that $f \equiv 0$?
Ungar called the following a ``freak theorem'':
The set of values of $\alpha$ for which \eqref{eq_int0} does not imply that $f \equiv 0$, is dense in $(0, \pi/2)$, and so is the set for which it does.

\subsection*{Acknowledgments}
The first author was supported by Grant RYC2021 - 031572 - I, funded by the Ministry of Science and Innovation / State Research Agency / 10.13039 / 501100011033 and by the Next Generation EU / Recovery, Transformation and Resilience Plan, and by Grant PID2022-136320NB-I00 funded by the Ministry of Science and Innovation.

The second author is supported in part by the U.S. National Science Foundation Grant DMS-2247771 and
the United States-Israel Binational Science Foundation (BSF).

\section{Analysis of the area equation}

\label{sec_area_equation}
Let us fix some notation first:
The canonical vectors of $\R^n$ are denoted by $e_i$, $i = 1, \ldots, n$.
The subspace generated by the vectors $v_1, \ldots, v_k$ is $\langle v_1, \ldots, v_k \rangle$.
The line in the plane going through the points $x$ and $y$ is written as $\overline{x, y}$ or $\overline{xy}$.
The centered unit Euclidean ball in $\R^k$ is $\B_k$ and its $k$-dimensional Lebesgue measure is $\omega_k$.
The only values we need are $\omega_3 = \frac 43 \pi$ and $\omega_2 = \pi$.
The centered circumference of radius $R$ will be denoted by $S_R$.

Let $C$ be as in Theorem \ref{thm_main}.
Since the convex body $C$ is assumed to be a symmetric body of revolution, we may assume without loss of generality that the axis of revolution is $e_1$, and $C$ it is completely determined by 
\begin{equation}
	\label{def_K}
	K = C \cap \langle e_1, e_2 \rangle. 
\end{equation}
This is a convex body in the plane $\langle e_1, e_2 \rangle$, which we identify naturally with $\R^2$, where the $X$ axis is the axis of revolution.
By the assumption in Theorem \ref{thm_main} that $C$ is a body of revolution, it suffices to consider only the $3$-dimensional planes $H$ in $\R^4$ that are parallel to $e_3$ and $e_4$.
Such a plane is determined by its intersection with $\langle e_1, e_2 \rangle$, which is a one-dimensional line tangent to the two-dimensional unit Euclidean ball $\B_2 \subseteq K \subseteq \R^2$.

In the case $C$ is the centered Euclidean ball of radius $R$, the area of intersection is 
\begin{equation}
	\label{eq_RA_relation}
	A = \omega_3 (R^2 - 1)^{3/2}.
\end{equation}
Then from the given area $A$ in Theorem \ref{thm_main}, our goal is to prove that $K$ is the Euclidean ball of radius $R>0$ given by the relation \eqref{eq_RA_relation}.

\subsection{Computation of the area of the section}
Since $K$ is convex and symmetric with respect to the $X$ axis, we can express the set $K$ as
\begin{equation}
	\label{eq_Ksubgrapf}
	K = \{(x,y) \in \R^2 : x \in [-X(K), X(K)], |y| \leq f(x)\},
\end{equation}
where $f:[-X(K),X(K)] \to \R_+$ is a concave function.
Here the interval of definition $[-X(K),X(K)]$ is the orthogonal projection of $K$ to the $X$ axis.
The number $X(K)$ will play an important role in the proof of Theorem \ref{thm_main}.
Notice that here we are not assuming necessarily that $f(x) \to 0$ as $x \to \pm X(K)$.
Furthermore, since $K$ is also symmetric with respect to the origin, then it is symmetric with respect to the $Y$ axis, and thus the function $f$ is even.

Given a line $L \subseteq \R^2$ which is tangent to $\B_2$ we want to compute the $3$-dimensional area of the intersection of $H_L = L + \langle e_3, e_4 \rangle$ with $C$.
This was done in \cite[Lemma 4.2]{alfonseca2025croft}. We reproduce it here for clarity.

Since $K$ is symmetric with respect to the $X$ axis, we can always assume that $L$ intersects $\B_2$ at the upper half-plane.
If $L$ is parallel to $e_2$, then the intersection of $C$ with $H_L$ is a $3$-dimensional ball of radius $f(\pm 1)$.
By \eqref{eq_RA_relation} we know that $f(\pm 1) = (R^2-1)^{1/2}$. 
By the symmetry assumption on $K$ we already know that there are at least four points in the intersection,
\[ \left(\pm 1, \pm \left(R^2-1\right)^{1/2} \right) \in S_R \cap \partial K.\]
If $L$ is not parallel to $e_2$ then let $s$ be its slope. Its equation in the plane is 
\begin{equation}
	\label{eq_line}
	y = s x + h(s), \text{ where } h(s) = \sqrt{1+s^2}.
\end{equation}
We will denote this line by $L_s$.

For $s\in \R$, the line $L_s$ intersects $\partial K$ in exactly two points. Let us call $a(s) < b(s)$ the $X$-coordinates of these points, so that 
\[\left\{ \left(a(s),f(a(s))\right), \left(b(s), f(b(s))\right) \right\} = \partial K \cap L_s.\]
The set $C$ is a union of $3$-dimensional balls $(\{x\} \times \langle e_2, e_3, e_4 \rangle ) \cap C$ with $x \in \R$.
The intersection of the $3$-dimensional hyperplane $H = L + \langle e_3, e_4 \rangle$ with $C$ is then a union of two-dimensional circles 
\[H \cap C = \bigcup_{a(s) < x < b(s)} H \cap (\{x\} \times \langle e_2, e_3, e_4 \rangle ) \cap C,\]
where each circle has radius
\[r_x =(  f(x)^2 - (h(s) + x s)^2)^{1/2}.\]
The area $A$ can be computed by integrating $\omega_2 r_x ^2$ with respect to $x$, with area element $\sqrt{1+s^2}$. This is
\[A = \int_{a(s)}^{b(s)} \sqrt{1+s^2}(f(x)^2 - (h(s) + x s)^2) \omega_2 dx.\]
Rearranging terms and taking derivative with respect to $s$ we get
\begin{align}
	\label{eq_A}
	\frac{A/\omega_2}{\sqrt{1+s^2}} &= \int_{a(s)}^{b(s)} (f(x)^2 - (h(s) + x s)^2) d x \\
	\frac{\partial}{\partial s}\left[ \frac{(R^2-1)^{3/2} \omega_3/\omega_2}{\sqrt{1+s^2}} \right] &= \int_{a(s)}^{b(s)} - 2 (h(s) + x s)(h'(s)+x) d x.
\end{align}
Notice that $f$ disappeared from the last equation (this is the only place in the paper where we use that $n=4$).

The case $s=0$ is trivial: Since $K$ is symmetric with respect to the $Y$ axis, this yields trivially $a(0) = - b(0)$.
For $s \neq 0$, we can divide the last equation by $s h'(s)$ (this will be useful for establishing the properties of the polynomial $P_s$ below).
In sum, we have just proved the following:
\begin{proposition}
	\label{prop_polynomialequation}
	Let $C$ be as in Theorem \ref{thm_main} and $K$ as in \eqref{def_K}.
	If $a(s)<b(s)$ are the $X$-coordinates of the two points of intersection of $\partial K$ with a line $L_s$ which is tangent to $\B_2$ at the upper half-plane and has slope $s$, then
	\begin{equation}
		\label{eq_H}
		P_{s}(b(s)) - P_{s}(a(s)) = - \frac{\frac 43 (R^2-1)^{3/2}}{s(1+s^2)},
	\end{equation}
	where $R$ is given by equation \eqref{eq_RA_relation} and
	\begin{equation}
		\label{eq_def_Ps}
		P_{s}(x) = -\frac 23 \frac{\sqrt{1+s^2}}{s} x^3 - \frac{1+2s^2}{s^2} x^2 -2 \frac{\sqrt{1+s^2}}{s} x.
	\end{equation}
	Notice that $P_{s}$ is a third-degree polynomial in $x$.
\end{proposition}

Consider $s \in \R$ and a tangent line $L_s$ with slope $s$ which touches the ball at a point $(t(s),d(s) )$ in the upper half-plane (see Figure \ref{fig_critical}), where $t(s), d(s)$ are functions of $s$.
For $s\neq 0$, the intersection point of $L_s$ with the $X$ axis is a point $(z(s),0)$, where $z(s)$ is also a function of $s$.
The functions $t(s),z(s)$ have very special properties, and will be used extensively in this section.
Observe that $z(s)$ and $t(s)$ have the same sign, which is opposite to the sign of $s$.
	\begin{figure}
		\caption{The values $t(s)$ and $z(s)$ are the critical points of the polynomial $P_s$, whose graph is depicted in blue.}
		\label{fig_critical}
		\includegraphics[width=\textwidth]{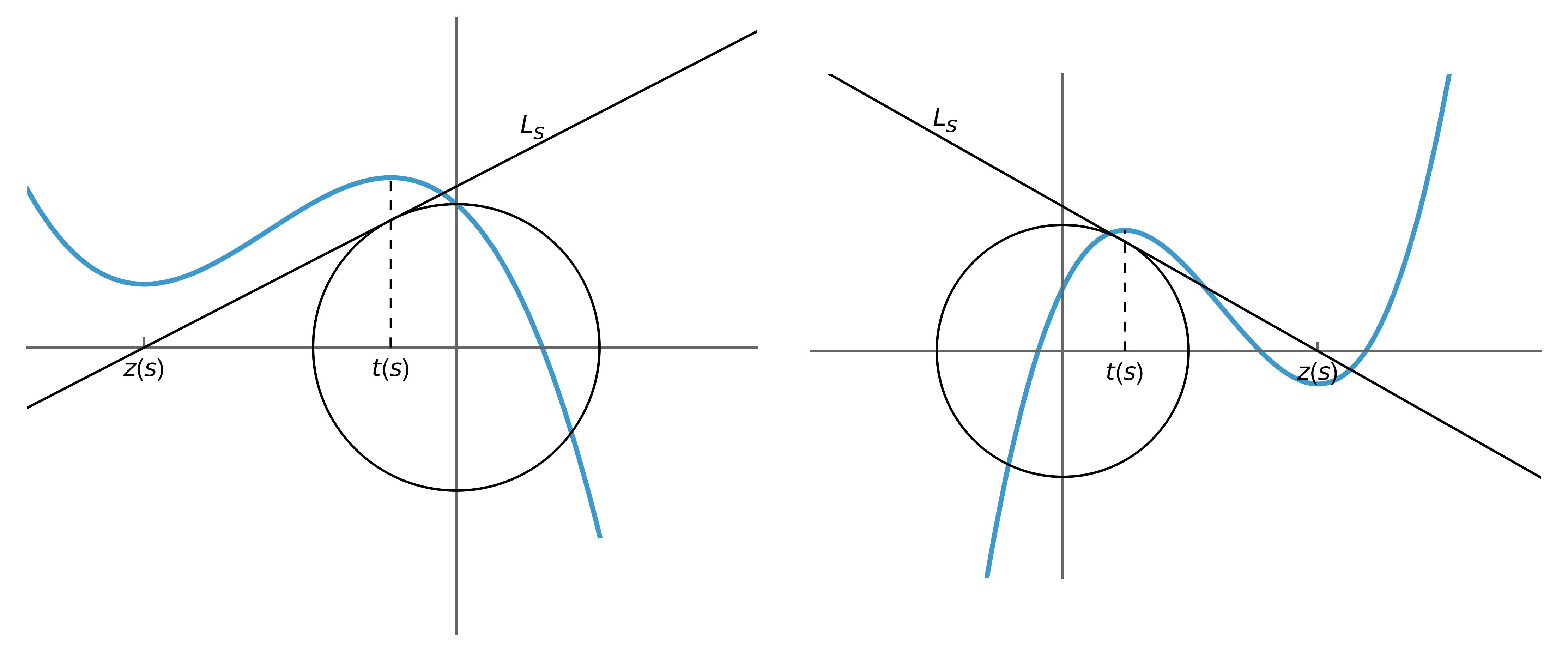}
	\end{figure}
Elementary computations show that
\begin{align}
	\label{eq_t}
	t(s) = -h'(s) = - \frac {s}{\sqrt{1+s^2}},\\
	\label{eq_z}
	z(s) = -h(s)/s = - \frac {\sqrt{1+s^2}}{s},
\end{align}
where $h$ is as in \eqref{eq_line}.
The following proposition is elementary.
\begin{proposition}
	\label{prop_zt}
	The functions $t(s), z(s)$ satisfy the following properties:
	\begin{enumerate}[label = (\alph*), ref = \alph* ]
		\item $t(s)$ and $z(s)$ are the only two critical points of $P_{s}$.
		\item $t(s)$ is always a local maximum of $P_{s}$ while $z(s)$ is always a local minimum.
		\item $z(s) t(s) = 1$. \label{eq_zt_product}
	\end{enumerate}
\end{proposition}

We determined the tangent line $L_s$ from the slope $s$, and from the fact that it touches the ball in the upper half-plane.
But we can also determine a tangent line passing through a given point in the plane, and fixing a ``sense of rotation''.
Given a point $(x,y) \in \R^2 \setminus \B_2$ there are two lines passing through $(x,y)$ and tangent to $\B_2$.
We choose $L_{(x,y)}$ to be the one such that $c y - d x > 0$, where ${(c,d)} = L_{(x,y)} \cap \B_2$.
In other words, if a point travels from $(x,y)$ to $(c,d)$ it will rotate clockwise with respect to the origin (see Figure \ref{fig_senseofrotation}).
The line $L_{(x,y)}$ is vertical if $x = \pm 1$, otherwise its slope is denoted by $s(x,y)$.
Observe that $d>0$ if and only if $x<c$.
If this happens, we define $\e(x,y) = 1$ to indicate that $L_{(x,y)}$ touches the ball in the upper half-plane.
On the contrary if $d<0$ and $x>c$,
we define $\e(x,y) = -1$, indicating that $L_{(x,y)}$ touches the ball in the lower half-plane.

Notice that the line we defined by $L_s$ with $s\in \R$ always touches the unit ball at the upper half-plane, and $z(s), t(s)$ are defined with respect to this line.
However, $L_p$ with $p \in \R^2 \setminus \B_2$ can touch the unit ball at the upper or lower half-plane, depending on the position of $p$ in the plane, and determined by the value of $\varepsilon(p)$, as in Figure \ref{fig_senseofrotation}.
For simplicity, if $\varepsilon(p)=1$ we will write $z(p) = z(s(p))$ and $t(p) = t(s(p))$.

If $p \in \partial K$, then $L_p \cap \partial K$ consists of exactly two points $\{p, q\}$.
We will set $T_K(p) = q$. Replacing $K$ with $R \B_2$, we define similarly $T_R = T_{R \B_2}$.

If $T_K(x,y) = (v,w)$ and $\e(x,y) = 1$, then we have
\begin{equation}
	\label{eq_middle}
	x < t(x,y) < v.
\end{equation}
This is, the first coordinate of the touching point of $L_{(x,y)}$ and $\B_2$, lies inside the interval of integration of \eqref{eq_A}.
	\begin{figure}
		\caption{The function $T_K$ for some points in the plane. Here $\varepsilon(p)=1, s(p)<0, \varepsilon(p')=-1, s(p')>0$ and $\varepsilon(p'') = -1, s(p'')>0$.}
		\label{fig_senseofrotation}
		\includegraphics[scale=.75]{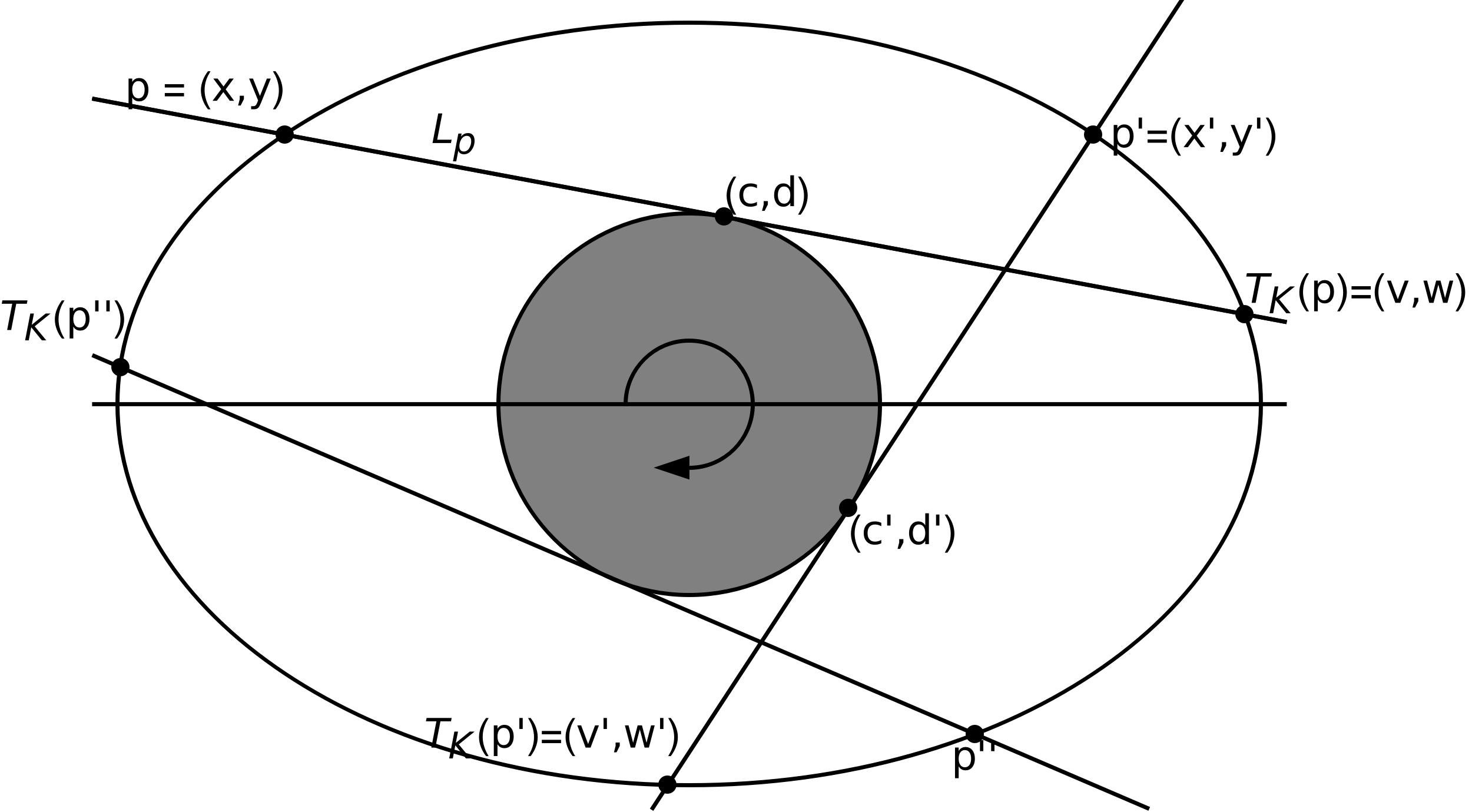}
	\end{figure}

Assume now that $C$ is as in Theorem \ref{thm_main}, $K$ is as in \eqref{def_K}, we have a point $(x,y) \in \partial K$ with $\e(x,y) = 1$ and we want to compute the point $(v,w) = T_K(x,y)$.
By Proposition \ref{prop_polynomialequation}, we know that $v$ must be a solution of the equation
\begin{equation}
	\label{eq_billardeq}
	P_{s}(v) = - \frac{\frac 43 (R^2-1)^{3/2}}{s(1+s^2)} + P_{s}(x) ,
\end{equation}
where $s=s(x,y)$, which is a third-degree polynomial equation on $v$. This means that if we fix $(x,y)$ there exist at most $3$ possible values of $v$ for which $T_K(x,y) = (v,w)$.
Assume additionally that $(x,y) \in S_R$ and take $(v',w') = T_R(x,y)$.
Since the convex body $R \B_4$ satisfies the same relation with the tangent planes as $C$ (this is, $R \B_4$ satisfies the hypothesis of Theorem \ref{thm_main}), we know that $v'$ is also a solution of equation \eqref{eq_billardeq}.

\subsection{Conditions for uniqueness of the next point}
Given a point $p \in \partial K$, equation \eqref{eq_billardeq} gives $3$ possibilities for the next point $T_K(p)$.
This represents a serious inconvenience, but there exist several situations for which \eqref{eq_billardeq} has only one solution which is valid (i.e. compatible with other properties of $K$).
That is the content of Proposition \ref{prop_uniqueness_cases}. But first we need a lemma.

\begin{lemma}
	\label{lem_side}
	Let $M \subseteq \R^2$ be a convex body, symmetric with respect to the $X$ and $Y$ axes, containing $\B_2$ in its interior, and let $X(M)$ be the largest $X$-coordinate of a point in $M$.
	Then for every $p \in \partial M$ with $\e(p)=1$ and $s(p)<0$, 
	the $Y$-coordinate of $T_M(p)$ has the same sign as $s(p) + s(X(M),0)$.
\end{lemma}
\begin{proof}
	First observe that $s(X(M),0) > 0$ and $\varepsilon(X(M),0) = -1$.
	Also, due to the symmetry of $M$ with respect to the $X$ axis, the slope $s_0$ of the line $L_{s_0}$ intersecting the $X$ axis in $X(M)$ is exactly $-s(X(M),0)$ (see Figure \ref{fig_verticalsymmetry}).
	In other words, $z(s_0) = X(M)$ where $s_0 = -s(X(M),0)$.
	\begin{figure}
		\caption{The lines $L_{s_0}$ and $L_{(X(M),0)}$ are symmetric with respect to the $X$ axis. If $s<s_0$ the segment of $L_s$ bounded by the red lines is between the $X$ axis and $L_{s_0}$.}
		\label{fig_verticalsymmetry}
		\includegraphics[scale=.75]{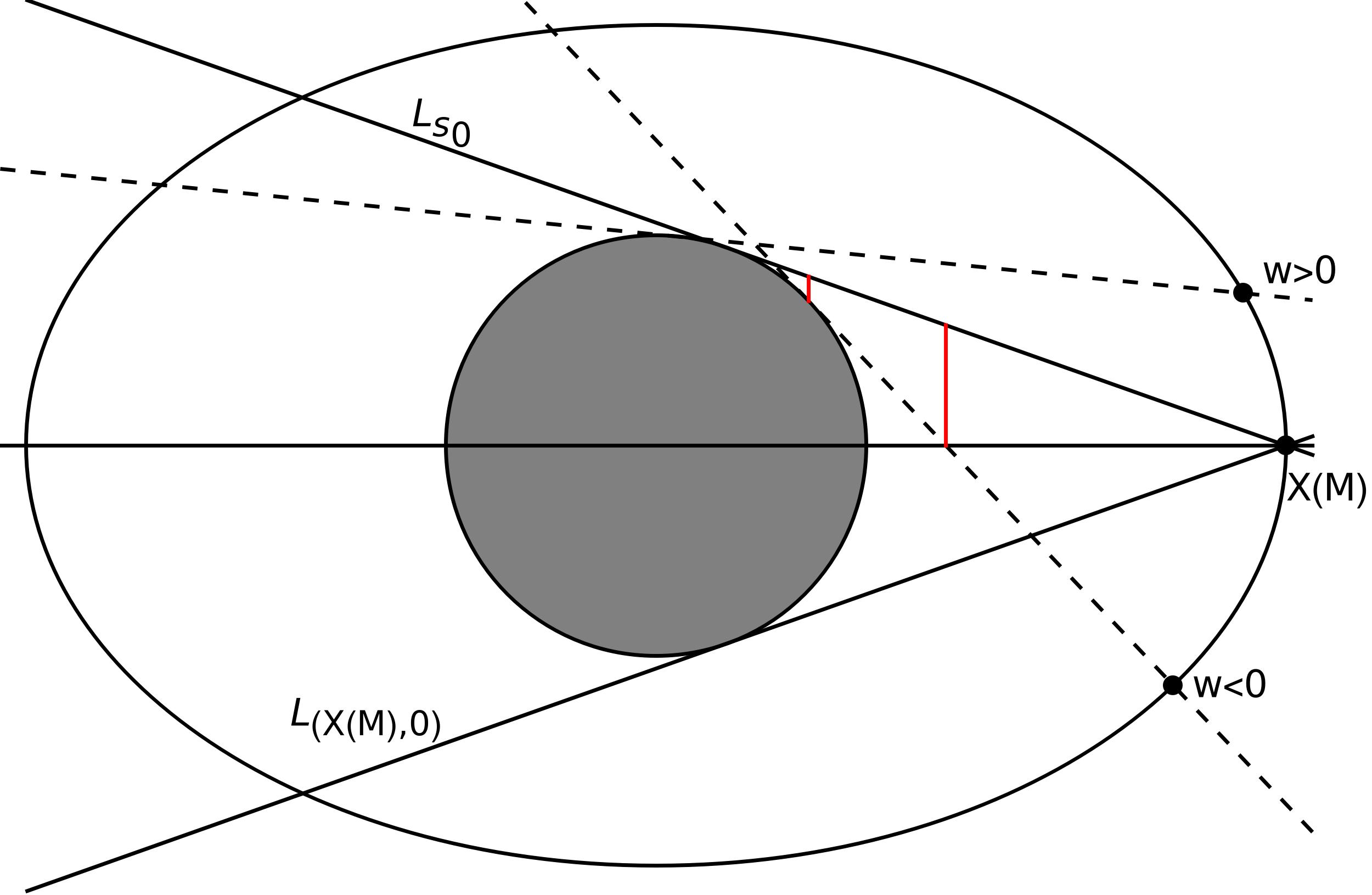}
	\end{figure}

	Letting $(v,w) = T_M(p)$, we must prove that $w > 0$ if $s(p) \in (s_0, 0)$, and $w<0$ if $s(p) < s_0$.

	Recall that the equation of the tangent line at the upper half-plane is given by \eqref{eq_line}.
	By taking derivative with respect to $s$, we see that $\xi s + h(s)$ is increasing with respect to $s$ for all $\xi > t(s)$.
	If $s(p) \in (s_0, 0)$, we have
	\[
		w = s(p) v + h(s(p)) > s(p) X(M) + h(s(p)) > s_0 X(M) + h(s_0) = 0.
	\]

	Now assume $s(p) < s_0$. 
	Since $z(s)$ and $t(s)$ are respectively increasing and decreasing with respect to $s$ for $s<0$, we know that
	\[( t(p), z(p) ) \subseteq (t(s_0), z(s_0)) = (t(s_0), X(M)).\]
	For any $\xi \in ( t(p), z(p) ) $ we have
	\[\xi s_0 + h(s_0) > \xi s(p) + h(s(p)) > 0.\]
	This implies that the segment 
	\[ \{(\xi, \xi s(p) + h(s(p))) :  \xi \in ( t(p), z(p) )\} \]
	of the line $L_{s(p)}$, is
	located between the segments
	\[ \{(\xi, \xi s_0 + h(s_0) ) :  \xi \in ( t(p), z(p) )\} \]
	and
	\[ \{(\xi, 0 ) :  \xi \in ( t(p), z(p) )\},\]
	so it is completely contained in the interior of $K$ (see Figure \ref{fig_verticalsymmetry}). Since we must have $(v,w) = T_K(p) \in \partial K$ and $v > t(p)$, we get $v > z(p)$. Then this implies
	\[
		w = s(p) v + h(s) < s(p) z(p) + h(s(p)) = 0,
	\]
	as we wanted to prove.

\end{proof}

\begin{proposition}
	\label{prop_uniqueness_cases}
	Let $K$ be defined by \eqref{def_K} where $C$ is as in Theorem \ref{thm_main}.
	Let $p \in S_R \cap \partial K$ be such that one of the following three conditions is satisfied:
	\begin{enumerate}[label = (\alph*), ref = \alph* ]
		\item \label{cond_region} $s(p) \geq 0$
		\item \label{cond_apriori} $3+2s(p) \sqrt{R^2-1} < 0$,
		\item \label{cond_interval} $s(p)$ is not between $-s(X(K),0)$ and $-s(R,0)$.
	\end{enumerate}
	Then $T_R(p) \in \partial K$.
\end{proposition}
\begin{proof}
	Since $K = -K$, $s(-p) = s(p)$ and $T_K(-p) = - T_K(p)$, we may assume always that $\e(p) = 1$.
	Let $(x,y) = p$.
	Take $(v,w) = T_K(x,y)$ and $(v',w') = T_R(x,y)$.
	Our goal in each condition is to prove that $T_K(p) = T_R(p)$, which implies that $T_R(p) \in \partial K$.
	This can be achieved by showing that $v = v'$, since both $(v,w)$ and $(v',w')$ belong to the line $L_p$.

	We begin with condition \eqref{cond_region}.
	If $s(p) = 0$ then it is clear that $v = v' = -x$ by the symmetry of $K$ with respect to the $Y$ axis.
	Now assume $s(p) > 0$.
	The third-degree polynomial $P_{s(p)}$ has two critical points, $t(p), z(p)$, of which $t(p)$ is a local maximum, and $z(p)$ is a local minimum.
	We recall that $x < t(p) < v,v'$ (see \eqref{eq_middle}), and by Property \eqref{eq_zt_product} in Proposition \ref{prop_zt}, we have $z(p) < -1 < t(p) <  0$.
	But this implies that $P_{s(p)}$ has no critical points in the interval $(t(p), \infty)$ making it an injective function (see Figure \ref{fig_critical}, left).
	Since both $v,v'$ are solutions of \eqref{eq_billardeq} in that interval, we conclude that $v=v'$ which means $T_R(p) = T_K(p)$, so the first part of the proof is done.

	Now let us consider condition \eqref{cond_apriori}.
	From the first part of the proof, we may assume $s(p) < 0$, which implies $0 < t(p) < z(p)$. From \eqref{eq_def_Ps} we see that
	\[\lim_{v \to \infty} P_{s(p)} (v) = +\infty.\]

	We claim that condition \eqref{cond_apriori} implies 
	\begin{equation}
		\label{eq_uniqueness_condition}
		P_{s(p)} (x) - \frac{\frac 43 (R^2-1)^{3/2}}{s(p)(1+s(p)^2)} > P_{s(p)} (t(p)).
	\end{equation}
	Given this claim, since $t(p)$ is the unique local maximum of $P_{s(p)} $, the polynomial $P_{s(p)} $ cannot have more than one pre-image of a point larger than $P_{s(p)} (t(p))$ (see Figure \ref{fig_critical}, right).
	Then again equation \eqref{eq_billardeq} has only one solution $v$ which forces $T_K(p) = T_R(p)$.
	Now let us prove the claim:

	Put $s = s(p)$. Since $p = (x,y) \in S_R$, we have
	\[
		x^2 + (x s + \sqrt{1+s^2})^2 = R^2.
	\]
	Solving for $x$ and considering $x < t(s) = - \frac s {\sqrt{1+s^2}}$ we find the value of $x$,
	\[
		x = - \frac s {\sqrt{1+s^2}} - \sqrt{\frac{R^2-1}{s^2+1}}.
	\]
	Replace this value of $x$ in \eqref{eq_uniqueness_condition} to obtain
	\begin{align}
		\frac{2 s^2}{3 (s^2+1)}-\frac{2 s^2+1}{s^2+1}+2 <
		-\frac{4 (R^2-1)^{3/2}}{3 s (s^2+1)}
		+\frac{2 \sqrt{s^2+1} }{s} \left(\sqrt{\frac{R^2-1}{s^2+1}} +\frac{s}{\sqrt{s^2+1}}\right) \\
		+\frac{2 \sqrt{s^2+1}}{3 s} \left(\sqrt{\frac{R^2-1}{s^2+1}}+\frac{s}{\sqrt{s^2+1}}\right)^3
		+\frac{2 s^2+1}{s^2} \left(\sqrt{\frac{R^2-1}{s^2+1}}+\frac{s}{\sqrt{s^2+1}}\right)^2 \\
	\end{align}
	which after some computations simplifies to \eqref{cond_apriori}, as required.
	
	Finally let us consider condition \eqref{cond_interval}.
	First notice that $-s(X(K),0), -s(R,0) < 0$.
	As before, we may assume $s(p) < 0$ and $\e(p) = 1$, which imply $x < t(p) < z(p)$.
	Since $t(p), z(p)$ are the only two critical points of $P_{s(p)} $, there can be at most one solution $v$ of \eqref{eq_billardeq} in each of the intervals
	\begin{equation}
		\label{eq_intervals}
		(-\infty, t(p)), (t(p), z(p)), (z(p), \infty).
	\end{equation}

	Once again, the solution in $(-\infty, t(p))$ has to be discarded, by \eqref{eq_middle}.
	Also, since $z(p) s(p) + h(s(p)) = 0$, the sign of $w = v s(p) + h(s(p))$ equals that of $z(p) - v$.
	Condition \eqref{cond_apriori} implies that $s(p) + s(X(K),0)$ and $s(p) + s(R,0)$ have the same sign.
	By Lemma \ref{lem_side} applied to $M=K$ and to $M=R \B_2$ we deduce that the signs of $w$ and $w'$ coincide, meaning that either $v,v' < z(p)$ or $v,v' > z(p)$.
	Since the interval $(-\infty, t(p))$ was discarded, $v, v'$ belong to the same interval in \eqref{eq_intervals} and are thus equal.
\end{proof}

From now on, the point in the plane $(\cos(\alpha), \sin(\alpha))$ will be denoted using the complex notation $e^{i \alpha}$.
\begin{proposition}
	\label{prop_small_interval}
	Assume we are given two points $R e^{i \alpha}, R e^{i \beta} \in S_R \cap \partial K$ with $0 < \alpha < \beta < \pi/2$.
	Then $R, X(K) \in [R \cos(\alpha), R \frac {\sin(\alpha-\beta)}{\sin(\alpha) - \sin(\beta)}]$.
\end{proposition}
\begin{proof}
	The number $x_0 = R \frac {\sin(\alpha-\beta)}{\sin(\alpha) - \sin(\beta)}$ is the $X$-coordinate of the intersection of the $X$ axis with the line $\overline {R e^{i \alpha}, R e^{i \beta}}$. By contradiction, if we had $X(K) > x_0$ this would imply that $R e^{i \alpha}$ is contained in the interior of the triangle with vertices $(0,0), R e^{i \beta}, (X(K),0)$, which is contained in $K$, and this would violate the fact that $R e^{i \alpha} \in \partial K$ (see Figure \eqref{fig_triangle}).
	\begin{figure}
		\caption{}
		\label{fig_triangle}
		\includegraphics[width=.75 \textwidth]{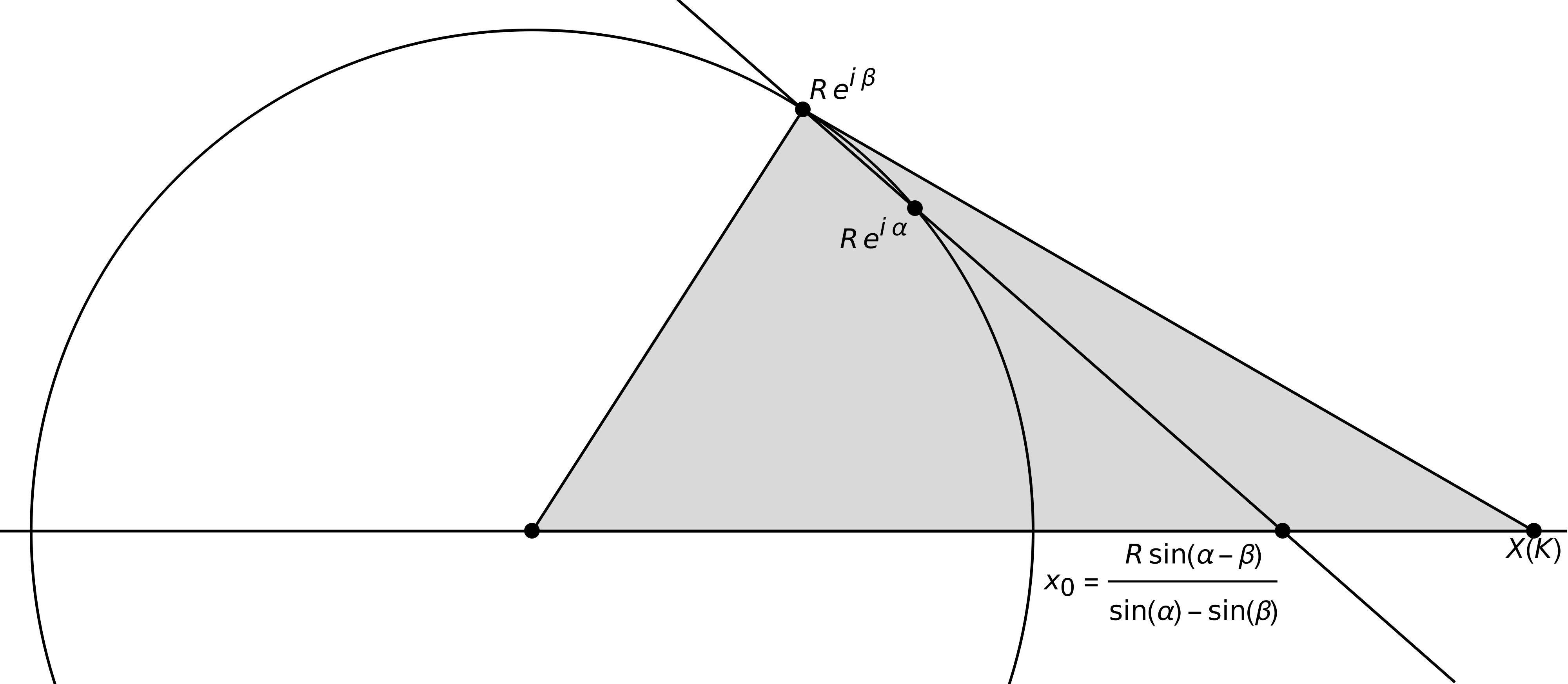}
	\end{figure}

	On the other hand, $X(K) \geq R \cos(\alpha)$ is evident from the fact that $X(K)$ is the maximal $X$-coordinate of points in $K$.

	The fact that $R$ belongs to the same interval is trivial.
\end{proof}

\section{Rotating around a shrinking interval}
\label{sec_rotation}

In this section we will define a sequence of points $u_k \in S_R$, and the goal is to show that $u_k \in S_R \cap \partial K$ for every $k \in \N$.
The starting point $u_1$ arises from the intersection of $C$ with one of the two ``vertical'' hyperplane (orthogonal to the axis of revolution), which is of course a Euclidean ball.
We choose to start with $u_1 = (1,-\sqrt{R^2-1})$.
Then we will combine Propositions \ref{prop_uniqueness_cases} and \ref{prop_small_interval} to show inductively that $u_k \in S_R \cap \partial K$ for every $k \geq 0$.

We need some definitions.
Recall that $A$, as in Theorem \ref{thm_main} represents the area of the intersection of $C \subseteq \R^4$ with any hyperplane tangent to $\B_4$.
All the quantities we define below will depend on $A$, so we will emphasize this dependence with the subscript $A$.
First of all define $R_A>0$ by
\[ A = \omega_3 (R_A^2 - 1)^{3/2}, \]
this is the relation \eqref{eq_RA_relation}.

Let
\begin{align}
\gamma_A
	&= \pi^{-1} \arctan((R_A^2-1)^{1/2}) \label{eq_def_gamma} \\
	&= \pi^{-1} \arctan((A/\omega_3)^{1/3}).
\end{align}
For $x > 1$ define
\[G_A(x) = \arctan(\sqrt{x^2-1}) - \arctan(\sqrt{R_A^2-1}).\]
For $\alpha, \beta \in (0,\pi/2)$ define
\[F_{A,+}(\alpha,\beta) = G_A \left( R_A \frac{\sin(\alpha - \beta)}{\sin(\alpha) - \sin(\beta)} \right),\]
\[F_{A,-}(\alpha) = -G_A(R_A \cos(\alpha)).\]
Also take 
\begin{align}
	\delta_{A,k}
	&= \pi \dist((2k-1) \gamma_A, \Z) \\
	&= \pi \min_{j \in \Z} |(2k-1) \gamma_A - j|.
\end{align}
Finally, for $k \geq 2$ we define $\alpha_{A,k} < \beta_{A,k}$ as the two lowest values of $\{ \delta_{A,1}, \ldots, \delta_{A,k} \}$. This is,
\begin{align}
	\label{eq_def_alphabeta}
	\alpha_{A,k} &= \min\{ \delta_{A,1}, \ldots, \delta_{A,k} \},\\
	\beta_{A,k} &= \min(\{ \delta_{A,1}, \ldots, \delta_{A,k} \} \setminus \{\alpha_{A,k} \}).
\end{align}

\begin{definition}
	\label{def_conditionC}
	Let $k \geq 2$. We say that $A>0$ satisfies condition $C(A,k)$, if the following inequalities hold.
	\begin{equation}
		\label{eq_verify}
		F_{A,+}(\alpha_{A,k}, \beta_{A,k}) < \delta_{A, k+1}, \quad F_{A,-}(\alpha_{A,k}) < \delta_{A, k+1}.
	\end{equation}
\end{definition}

Let us explain the quantities we just defined.
Let $x>1$ and consider a line $L$ passing through $(x,0)$, intersecting $\B_2$ at the upper half-plane.
This line intersects $S_{R_A}$ at two points, one of them is $R_A e^{i \eta}$ with $\eta \in (-\pi/2, \pi/2)$ (see Figure \ref{fig_projection}). Then the function $G_A$ computes exactly this angle $\eta$ as a function of $x$. This is, $\eta = G_A(x)$, as shown in Figure \ref{fig_projection}.

	\begin{figure}
		\caption{The function $G_A$ is computed as the difference between the green angle $\arctan(\sqrt{x^2-1})$ and the red angle, given by $\arctan(\sqrt{R_A^2-1})$.}
		\label{fig_projection}
		\includegraphics[width=\textwidth]{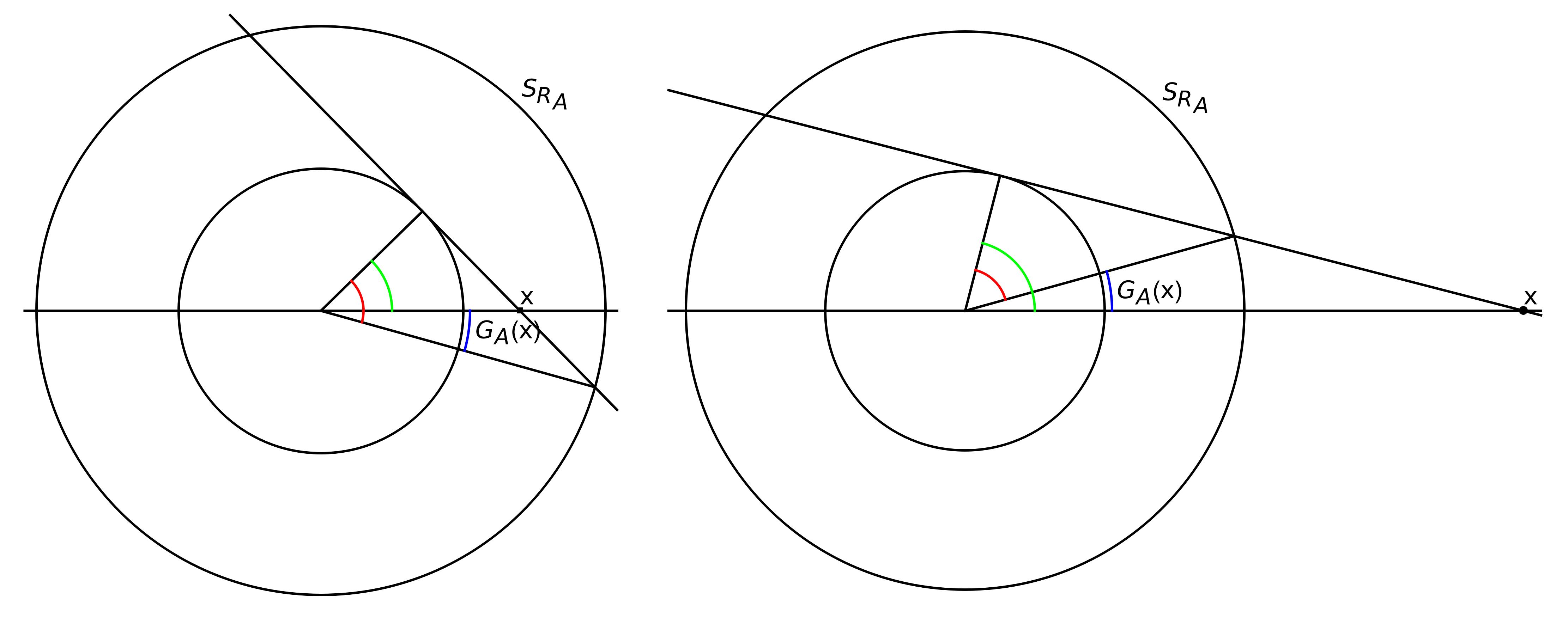}
	\end{figure}
	Regarding the numbers $\delta_{A,k}, \alpha_{A,k}, \beta_{A,k}$, we consider the points $u_{A,k} \in S_{R_A}, k \geq 1$ given by
\begin{equation}
	\label{eq_def_uk}
	u_{A,k} = R_A e^{i \pi (1-2k) \gamma_A}.
\end{equation}
It is important to notice that $\delta_{A,k}$ is exactly the angle in $(0, \pi/2)$ formed by the $X$ axis and the line $\overline{0, u_{A,k}}$.
The values of $\alpha_{A,k}, \beta_{A,k}$ are the angles corresponding to the two points from $\{u_{A,1}, \ldots, u_{A,k} \}$ which are closest to the $X$ axis. They will be used to bound the distance from $X(K)$ to $R_A$, as shown in Figure \ref{fig_interval}.
	\begin{figure}
		\caption{The points $u_{A,1}, \ldots, u_{A,k}$ bound the position of $X(K)$ which belongs to the red segment.
		The orange arc contains all the points $p \in S_R$ in the upper half-plane, with $s(p)$ between $-s(X(K),0)$ and $-s(R,0)$.
		The blue arc is the points with angle $[-F_{A,-}(\alpha_k), F_{A,+}(\alpha_k, \beta_k)]$, as in condition $C(A,k)$.
	}
		\label{fig_interval}
		\includegraphics[width=\textwidth]{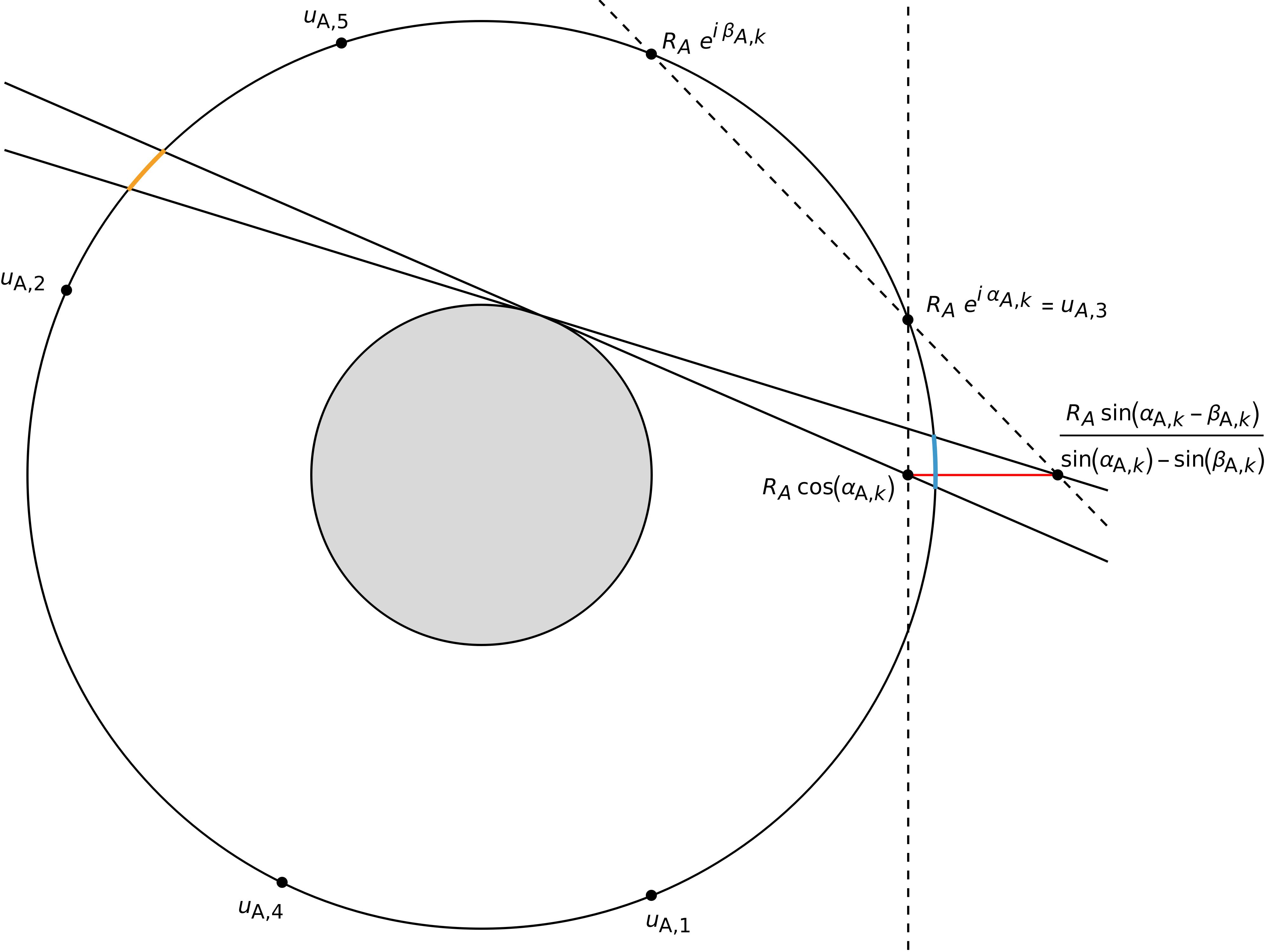}
	\end{figure}
	Now we are in conditions to define the set $\mathbb A$:
\begin{definition}
	\label{def_A}
	We define $\mathbb A \subseteq \R$ as the set of values $A>8 \omega_3$ such that $\gamma_A$ is irrational and condition $C(A,k)$ holds for every $k \geq 2$.
\end{definition}
\begin{theorem}
	\label{thm_infinite_rotation}
	Let $C \subseteq \R^4$ be a symmetric convex body of revolution containing the unit ball $\B_4$, with the following property:
	For every hyperplane $H$ tangent to $\B_4$, $\vol[3]{C \cap H} = A$ where $A$ is a constant independent of $H$.
	If $A \in \mathbb A$ then $C$ is a Euclidean ball.
\end{theorem}

\begin{proof}
	Take $A \in \mathbb A$.
	For notational convenience, in this proof we shall omit the subindex $A$ and write simply $\gamma = \gamma_A, R = R_A, G = G_A, \delta_k = \delta_{A,k}$ and so on.
	Take the tangent plane $H = \{1\} \times \langle e_2, e_3, e_4 \rangle $. By the hypothesis on $C$, $H \cap C$ is a Euclidean ball of area $A$, thus, of radius $(R^2-1)^{1/2}$.
	This means that the point $u_1 = (1, - (R^2-1)^{1/2})$ belongs to $S_R \cap \partial K$.
	This point is
	\[u_1 = (1,-(R^2-1)^{1/2}) = R e^{-i \pi \gamma},\]
	where $\gamma = \gamma_A$ is given by \eqref{eq_def_gamma}, and $u_j = u_{A,j}$ are as in \eqref{eq_def_uk}.

	Let us prove that the point $u_1$ satisfies condition \eqref{cond_apriori} in Proposition \ref{prop_uniqueness_cases}.
	Using equation \eqref{eq_line} for the lower half-plane,
	\[s(u_1) -\sqrt{s(u_1)^2+1} = -(R^2 -1)^{1/2},\]
	we may compute $s(u_1)$ obtaining
	\[s(u_1) = \frac{2-R^2}{2(R^2-1)^{1/2}}.\]
	By hypothesis, $R^2-1 = (A/\omega_3)^{2/3} > 4$, and we get
	\[3+2s(u_1) \sqrt{R^2-1} = 3 + (2-R^2) < 0.\]
	Then by Proposition \ref{prop_uniqueness_cases}, $u_2 = T_R(u_1)$ belongs to $S_R \cap \partial K$.

	Now we proceed by induction and assume that $u_1, \ldots, u_k \in S_R \cap \partial K$ for some $k \geq 2$.
	We intend to prove that $u_{k+1} \in S_R \cap \partial K$.

	As mentioned, the angle formed by the line $\overline{ 0, u_j }$ and the $X$ axis is exactly $\delta_j$.
	The point $u_j$ coincides with $e^{i \delta_j}$ only when $u_j$ belongs to the first quadrant. However, due to the symmetry of $K$ with respect to the $X$ and $Y$ axes, we still have $R e^{i \delta_j} \in S_R \cap \partial K$ for $j=1, \ldots, k$.

	Recall that $0< \alpha_k < \beta_k$ are the two lowest values of $\{\delta_1, \ldots, \delta_k\}$ (see \eqref{eq_def_alphabeta}).
	Since $k \geq 2$, $\alpha_k$ and $\beta_k$ are well defined. They are also different due to the irrationality of $\gamma$.
	Taking $\alpha = \alpha_k$ and $\beta = \beta_k$ in Proposition \ref{prop_small_interval}, 
	\begin{equation}
		\label{eq_Rl_interval}
		R,X(K) \in [R \cos(\alpha_k), R\frac {\sin(\beta_k-\alpha_k)}{\sin(\beta_k) - \sin(\alpha_k)}] .
	\end{equation}

	If $s(u_k) \geq 0$ then Proposition \ref{prop_uniqueness_cases} condition \eqref{cond_region} implies $u_{k+1} = T_R(u_k) \in S_R \cap \partial K$, so we may assume that $s(u_k) < 0$.
	Assume for simplicity that $\e(u_k) = 1$, the case $\e(u_k) = -1$ being analogous.
	We intend to show that condition \eqref{cond_interval} in Proposition \ref{prop_uniqueness_cases} is satisfied.

	The geometric interpretation of the function $G$ before the proof shows that if $p \in S_R$ is a point with $\varepsilon(p) = 1$ and $s(p)<0$ (and thus $z(p)>1$),
	\[T_R(p) = R e^{i \alpha}, \text{ with } \alpha = G(z(p)) \in (-\pi/2, \pi/2) .\]
	By \eqref{eq_Rl_interval} and the monotonicity of $G$, 
	\begin{align}
		\label{eq_GInterval}
		G(R), G(X(K)) 
		&\in [G(R \cos(\alpha_k)), G(R\frac {\sin(\beta_k-\alpha_k)}{\sin(\beta_k) - \sin(\alpha_k)})] \\
		&= [-F_-(\alpha_k), F_+(\alpha_k,\beta_k)].
	\end{align}
	Assume by contradiction that $s(u_k)$ is between $-s(X(K),0),-s(R,0)$. Then, $G$ and $z$ being monotone, $G(z(u_k)) = G(z(s(u_k)))$ must be located between $G(X(K))$ and $G(R)=0$ (see Figure \ref{fig_interval}), and consequently $\delta_{k+1} = |G(z(u_k))| \leq |G(X(K))|$.
	By \eqref{eq_GInterval} we get $\delta_{k+1} \leq \max\{F_-(\alpha_k), F_+(\alpha_k,\beta_k)\}$, which contradicts that condition $C(A,k)$ is satisfied. This completes the induction step.

	We just proved that the full sequence $\{u_1, u_2, \ldots \}$ is inside $S_R \cap \partial K$, which forms a dense orbit in $S_R$ since $\gamma$ is an irrational number.
	By the compactness of $\partial K$, we conclude that $S_R \subseteq \partial K$, and due to the convexity of $K$, this finally implies that $K = R \B_2$, and $C = R \B_4$.

\end{proof}

\section{Reduction to a finite number of inequalities}
\label{sec_finite_inequalities}
The conditions $C(A,k)$ in Theorem \ref{thm_infinite_rotation} consist of an infinite number of inequalities, but they can be reduced to a finite number using arithmetic properties of $\gamma_A$.
That is the objective of the present section.

Let us explain this situation:
The points $u_{A,k}$ form the orbit of an irrational rotation (of angle $\pi \gamma_A$ with irrational $\gamma_A$) around $S_R$.
As $k \to \infty$ the orbit becomes more and more dense, and we find points that are arbitrarily close to the $X$ axis, thus $\alpha_{A,k}, \beta_{A,k} \to 0$.
The inequalities \eqref{eq_verify} define a small interval of angles $[-F_{A,-}(\alpha_{A,k}), F_{A,+}(\alpha_{A,k},\beta_{A,k})]$ that shrinks to $\{0\}$ as $k \to \infty$.
Moreover, $F_\pm$ are functions that tend to $0$ quadratically, meaning that $0 < F_{A,-}(\alpha), F_{A,+}(\alpha, \beta) \leq c (\alpha^2 + \beta^2)$ for small $\alpha, \beta > 0$ and some constant $c>0$.
Then the induction argument of Theorem \ref{thm_infinite_rotation} can continue as long as the angle $\delta_{A,k}$ of $u_{A,k}$ does not fall in that interval.

	For the sake of clarity let us simplify the elements at play. Take any irrational number $\gamma$ and consider its multiples $\{\gamma, 2 \gamma, \ldots, k\gamma \}$. As $k \to \infty$ some of their fractional parts will approach $0$ or $1$.
	Lets assume we need to verify a condition of the form
	\begin{equation}
		\label{eq_simpler_C}
		g(\min_{j\leq k} \dist(j \gamma, \Z)) < \dist((k+1) \gamma, \Z), \text{ for all } k \geq k_0,
	\end{equation}
	where $g$ tends to $0$ quadratically.
	The Dirichlet theorem on Diophantine approximation states that for every $k \geq 1$ there exist natural numbers $j,l$ with $j \leq k$ such that $|j \gamma - l| \leq k^{-1}$.
	Thus, the left-hand side of \eqref{eq_simpler_C} decays roughly at speed $c k^{-2}$.
	If $\gamma$ is ``sufficiently irrational'', for example if it is an algebraic number, then approximations of $\gamma$ by rational numbers $p/q$ will fail by roughly $c_{\gamma,\varepsilon} q^{-2-\varepsilon}$, by Roth's theorem, and this implies that $\dist(k \gamma,\Z)$ is larger than $c_{\gamma,\varepsilon} k^{-1-\varepsilon}$.
	Then clearly the inequalities in \eqref{eq_simpler_C} will be verified for $k \geq k_0$.
	Unfortunately our situation is significantly more complicated: In the first place we have $\delta_{A,k} = \pi \dist((2k-1) \gamma, \Z)$, so only odd multiples of $\gamma$ have to be considered, and that prevents us from using the Dirichlet theorem.
	On the second place, there is the issue of the constant $c_{\gamma, \varepsilon}$, which not only can be difficult to bound, but also can tend to infinity as $\gamma$ tends to a rational number. In our case, for example, when $A \to \infty$ we have $\gamma \to 1/2$ (see equation \eqref{eq_def_gamma}).
	We need then to make a finer analysis of the numbers $\alpha_{A,k}, \beta_{A,k}, \delta_{A,k}$, and it turns out that the information we need can be read in the continued fraction of $\gamma_A$.

\subsection{Arithmetic properties of the rotation number}
\label{sec_arithmetic}

Every positive irrational number $\gamma$ can be written uniquely as an infinite continued fraction:
\[\gamma = a_0 + \frac 1{a_1 + \frac 1 {a_2 + \frac 1 {a_3 + \cdots}}},\]
where the numbers $a_i \in \N$ are called the coefficients.
We refer to Khinchin's book \cite{khinchin1964continued} for the basic theory of continued fractions.
We use the notation 
\[\gamma = [a_0, a_1, a_2, \ldots ].\]
Sometimes the coefficients are written as $a_j(\gamma)$, but we will not write the dependence on $\gamma$ to avoid overcharging the notation.

Truncating this expression we obtain a rational number
\[ \frac{p_j}{q_j} = [a_0, \ldots, a_j] = a_0 + \frac 1{a_1 + \frac 1{ \cdots + \frac 1{a_j} } }, \]
which is an irreducible fraction converging to $\gamma$ as $j \to \infty$.
The fractions $p_j/q_j$ are called the convergents of $\gamma$ and they can be defined by the recursions
\begin{align}
	q_{-1} = 0, q_0 = 1, \quad q_j = a_j q_{j-1} + q_{j-2},	\label{eq_recursion_q} \\
	p_{-1} = 1, p_0 = a_0, \quad p_j = a_j p_{j-1} + p_{j-2}. \label{eq_recursion_p}
\end{align}
The first significant fraction is $\frac {p_0}{q_0} = \frac {a_0}1$ whereas the first values $p_{-1}, q_{-1}$ are artificially set, for convenience (see \cite[equation (7)]{khinchin1964continued}).

The denominators $k = q_j$ play a significant role here, as they are exactly the places at which $\delta_{A,k}$ is smaller than all their predecessors.
This is the consequence of the following property:
\begin{proposition}[{\cite[Theorem 17]{khinchin1964continued}}]
	\label{prop_bestapprox}
	If $\gamma$ is an irrational number and $p_k/q_k$ are its convergents, then for every $p,q \in \N$ with $q \leq q_k$,
	\[|q \gamma - p| \geq |q_k \gamma - p_k|.\]
\end{proposition}

Moreover, we have a good control of the distance from $q_k \gamma$ to the natural numbers.
\begin{proposition}
	\label{prop_estimate}
	If $\gamma$ is an irrational number and $p_k/q_k$ are its convergents, then for every $k \in \N$,
	\[ \frac 1{q_k} \frac 1{a_{k+1} +2} < |\gamma q_k - p_k| < \frac 1{q_{k+1}}.\] 
\end{proposition}
\begin{proof}
	The right inequality is \cite[Theorem 9]{khinchin1964continued}.
	The left inequality is a combination of \cite[Theorem 13]{khinchin1964continued} and \eqref{eq_recursion_q}.
\end{proof}

It is known that the denominators $q_k$ grow to infinity at least at an exponential rate. We have,
\begin{proposition}[{Theorem 12, \cite{khinchin1964continued}}]
	\label{prop_expgrowth}
	For any positive irrational $\gamma$ and $j \geq 2$,
	\[q_j \geq 2^{\frac{j-1}2}.\]
\end{proposition}

The definition of $\gamma_A$ and the inequality $A > 8 \omega_3$ restrict $\gamma_A$ to the interval $(\arctan(2)/\pi, 1/2)$, where $\arctan(2)/\pi > 1/3$.
This forces the first two coefficients of $\gamma_A$ to be $a_0(\gamma_A) = 0$, $a_1(\gamma_A) = 2$.
The property of $q_k$ given in Proposition \ref{prop_bestapprox} is useful only when $q_k$ is always an odd number. This can be achieved by choosing the correct parity on the coefficients.
The following is a direct consequence of the recursion formula \eqref{eq_recursion_q}.
\begin{proposition}
	\label{prop_parity}
		Assume $a_0=0$, $a_1=2$, $a_3$ is odd and $a_k$ is even for every $k \geq 4$.
		Then $q_j$ is odd for every $j \geq 2$.
		The first values are $q_0 = 1, q_1 = 2$.
\end{proposition}
In this case, we define
\begin{equation}
	\label{eq_tilde}
	\tilde q_j = \frac 12(q_j + 1) \in \N
\end{equation}
for $j \geq 2$.

The following theorem aims to be used to verify condition $C(A,k)$ for every $k \geq 2$ using only finitely many inequalities that involve the coefficients $a_j$.
	\begin{theorem}
		\label{thm_estimates}
		Let $a_j$ be the coefficients in the continued fraction of $\gamma_A$.
		Assume $a_0=0$, $a_1=2$, $a_3$ is odd and $a_k$ is even for every $k \geq 4$.
		Fix some $j_0 \geq 3$.

		\begin{enumerate}[label = (\alph*), ref = \alph* ]
			\item
			\label{eq_conditions_1}
			If 
				\[\frac{\pi}{\cos(\pi / q_j) \sqrt{R_A^2 \cos(\pi / q_{j+1})^2 - 1}} (a_{j+2}+2) \leq q_j\]
			 for every $j \geq j_0$, then
			$C(A,k)$ holds for every $k \geq \tilde q_{j_0}$.

			\item 
			\label{eq_conditions_2}
			If $j_0 \geq 4$ and 
			\[\max\{F_{A,+}(\delta_{\tilde q_{j-1}}, \delta_{\tilde q_{j-2}}), F_{A,-}(\delta_{\tilde q_{j-1}})\} < \delta_{\tilde q_j}\]
			for $j = 4, \ldots, j_0$ then $C(A,k)$ holds for $k \in [\tilde q_3, \tilde q_{j_0})$.

			\item 
			\label{eq_conditions_3}
			If $3 \leq \tilde q_2 \leq \tilde q_3-2$ and $C(A,\tilde q_2-1)$ holds then $C(A,k)$ holds for $k \in [\tilde q_2, \tilde q_3-2]$.
		\end{enumerate}
	\end{theorem}
Before the proof we make some observations:
The condition \eqref{eq_conditions_1} consists of infinitely many inequalities.
But by the exponential growth of $q_j$ (Proposition \ref{prop_expgrowth}), they will be automatically satisfied for large $j_0$ if we impose a sub-exponential (polynomial or even constant) growth condition on the coefficients $a_j$.
This is not a strong requirement. For example, by \cite[Theorem 30]{khinchin1964continued}, for almost all real numbers $x$, there is $j_x \in \N$ such that $a_j \leq j^2$ for $j \geq j_x$.
Also, it is known that the sequence $a_j$ is eventually periodic (and thus bounded) if and only if $\gamma_A$ is the root of a quadratic polynomial.

For numbers $\gamma$ with bounded continued fraction it suffices to take the first $j_0$ for which
\begin{equation}
	\label{eq_firstj0}
	\frac{\pi}{\cos(\pi / q_{j_0}) \sqrt{R_A^2 \cos(\pi / q_{j_0+1})^2 - 1}} \left(\max_{j \geq j_0+2}\{a_j\}+2\right) \leq q_{j_0},
\end{equation}
and condition \eqref{eq_conditions_1} will be automatically satisfied, since the left-hand side is decreasing.

Using Proposition \ref{prop_expgrowth} and $R_A^2>5$, it is easy to see that taking $j_0 \geq 5$, the inequality
\begin{equation}
	\label{eq_simple_condition}
	\pi (a_{j+2}+2) \leq 2^{\frac{j-1}2}
\end{equation}
for $j \geq j_0$, implies condition \eqref{eq_conditions_1}.
In other words, if $\gamma_A$ has a continued fraction with sub-exponential growth, then $C(A,k)$ will be satisfied for large $k$.

If all conditions \eqref{eq_conditions_1}, \eqref{eq_conditions_2} and \eqref{eq_conditions_3} are verified, it remains to check $C(A,k)$ for $k \in [2, \tilde q_2 - 2]$ and $k=\tilde q_3 -1$.
For some examples it is enough to verify condition \eqref{eq_conditions_1} for $j \geq j_0$ and check the remaining conditions $C(A, 2), \ldots, C(A, \tilde q_{j_0}-1)$ individually. If $\tilde q_{j_0}$ is a very large number, condition \eqref{eq_conditions_2} can be used to reduce it to $\tilde q_3$, and condition \eqref{eq_conditions_1}, to $\tilde q_2$.
Notice that condition \eqref{eq_conditions_2} is stronger than $C(A,\delta_{\tilde q_{j-1}})$, because $\tilde q_{j-2}$ is not necessarily the second smallest value of $\{\delta_1, \ldots, \delta_{\tilde q_{j-1}} \}$. The advantage of putting $\delta_{\tilde q_{j-2}}$ instead of $\beta_{\tilde q_{j-1}}$ is that the computation of $\beta$ can be very intensive if $\tilde q_{j_0}$ is very large.
	\begin{proof}[Proof of Theorem \ref{thm_estimates}]
		As in the proof of Theorem \ref{thm_infinite_rotation}, we shall omit the subindex $A$ and write simply $\gamma = \gamma_A, R = R_A, G = G_A, \delta_k = \delta_{A,k}$ and so on.
		Recall that
		\[\delta_k = \pi \inf_{m \in \N_0} |(2k-1) \gamma - m|.\]
		Since by Proposition \ref{prop_parity}, all $q_j$ are odd for $j \geq 2$, we define $\tilde q_j$ by \eqref{eq_tilde}. By Proposition \ref{prop_bestapprox},
		\begin{align}
			\delta_{\tilde q_j} 
			&= \pi \inf_{m \in \N_0} |q_j \gamma - m| \\
			&= \pi |q_j \gamma - p_j|.
		\end{align}

		Let $k \in \N$, $k \geq \tilde q_{j_0}$, and take $j$ such that $k \in [\tilde q_j,\tilde  q_{j+1})$ with $j \geq j_0$.
Propositions \ref{prop_bestapprox} and \ref{prop_estimate} imply that
		\begin{equation}
		\label{eq_alpha_bound}
			\alpha_k = \min\{\delta_1, \ldots, \delta_k\} = \delta_{\tilde q_j} = \pi |q_j \gamma - p_j| \leq \frac \pi {q_{j+1}}.
		\end{equation}
		Moreover, 
		\begin{equation}
		\label{eq_beta_bound}
			\beta_k \leq \pi |q_{j-1} \gamma - p_{j-1}| \leq \frac \pi {q_j}.
		\end{equation}

		(Notice that $\beta_k$ is the second lowest value of $\{\delta_1, \ldots, \delta_k\}$, which is not necessarily attained at $\delta_{\tilde q_{j-1}}$, thus the left inequality).

		We intend to show that inequalities \eqref{eq_verify} are satisfied.
		We start with $F_-$. Since $G(R)=0$ and $G'(x) = \frac 1 {x\sqrt{x^2-1}}$ is positive and decreasing for $x>1$,
		\begin{align}
			F_-(x)
			&= -G(R \cos(x)) \\
			&\leq \frac 1{R \cos(x) \sqrt{(R \cos(x))^2 - 1}} (R - R \cos(x) ) \\
			&\leq \frac 1{\cos(x) \sqrt{(R \cos(x))^2 - 1}} x^2.
		\end{align}
		Since $F_-$ is increasing, we bound
		\begin{align}
			F_-(\alpha_k)
			&\leq F_-(\pi/q_{j+1}) \\
			&\leq \frac {\pi^2}{\cos(\pi/q_{j+1}) \sqrt{R^2 \cos(\pi/q_{j+1})^2-1}} \frac 1{q_{j+1}^2}. \label{eq_fminusbound}
		\end{align}
		For $F_+$, since $R \frac{\sin(s - t)}{\sin(s) - \sin(t)} \geq R$ and $G$ is concave, we have
		\begin{align}
                        F_+(s,t)
			&=G \left( R \frac{\sin(s - t)}{\sin(s) - \sin(t)} \right) \\
			&\leq G'(R) \left(R \frac{\sin(s - t)}{\sin(s) - \sin(t)} - R \right) \\
			&\leq \frac 1 {\sqrt{R^2-1}} \left( \frac{\sin(s - t)}{\sin(s) - \sin(t)} - 1 \right) \\
			&= \frac 1 {\sqrt{R^2-1}} \frac{2 \sin(s/2) \sin(t/2)}{\cos(\frac {s+t}2)} \\
			&\leq \frac 1 {\sqrt{R^2-1}} \frac{s t}{2 \cos(\frac {s+t}2)},
		\end{align}
		so
		\begin{align}
			F_+(\frac \pi{q_{j+1}}, \frac \pi {q_j})
			&\leq \frac 1 {\sqrt{R^2-1}} \frac {\pi^2} {2 \cos(\pi / q_j)} \frac 1{q_{j+1} q_j} \label{eq_fplusbound}.
		\end{align}
		On the other hand, $k+1 \leq \tilde q_{j+1}$. By the left inequality of Proposition \ref{prop_estimate},
		\begin{align}
			\delta_{k+1}
			&\geq \pi |q_{j+1} \gamma - p_{j+1}| \\
			&\geq \frac \pi {q_{j+1} (a_{j+2} + 2)} \label{eq_rlower}.
		\end{align}
		Putting together \eqref{eq_alpha_bound}, \eqref{eq_beta_bound}, \eqref{eq_fminusbound}, \eqref{eq_fplusbound}, condition \eqref{eq_conditions_1}, and \eqref{eq_rlower}, we get 

		\begin{align}
			\max\{F_{+}(\alpha_{k}, \beta_{k}), F_{-}(\alpha_{k}) \}
			&\leq \max \left\{ \frac {\pi^2}{\cos(\pi/q_{j+1}) \sqrt{R^2 \cos(\pi/q_{j+1})^2-1}} \frac 1{q_{j+1}^2},\right.\\ &\left. \frac 1 {\sqrt{R^2-1}} \frac {\pi^2} {2 \cos(\pi / q_j)} \frac 1{q_{j+1} q_j} \right\} \\
			& \leq  \frac {\pi^2}{\cos(\pi/q_j) \sqrt{R^2 \cos(\pi/q_{j+1})^2-1}} \frac 1{q_{j+1} q_j} \\
			& \leq \frac \pi {q_{j+1} (a_{j+2} + 2)} \\
			& \leq \delta_{k+1},
		\end{align}
		which is exactly condition $C(A,k)$.

		We have so far established condition $C(A,k)$ for $k \geq \tilde q_{j_0}$.
		Let us now prove \eqref{eq_conditions_2}.
		Take any $k \in \N$ with $k \in [\tilde q_3, \tilde q_{j_0})$, and $j$ such that $k \in [\tilde q_{j-1}, \tilde q_j)$ with $j \in [4,j_0)$.
		Once again, $k+1 \leq \tilde q_j$ so $\delta_{k+1} \geq \delta_{\tilde q_j}$.
		Also, $j \geq 4$ which guarantees that $q_{j-1}, q_{j-2}$ are odd, and
		\[\alpha_k = \delta_{\tilde q_{j-1}}, \beta_k \leq \delta_{\tilde q_{j-2}}.\]
		Since $F_\pm$ are increasing with respect to all their coordinates, using the condition in \eqref{eq_conditions_2},
		\[
			\max\{F_+(\alpha_k, \beta_k), F_-(\alpha_k)\}
			\leq \max\{F_+(\delta_{\tilde q_{j-1}}, \delta_{\tilde q_{j-2}}), F_-(\delta_{\tilde q_{j-1}}) \} 
			\leq \delta_{\tilde q_j} 
			\leq \delta_{k+1},
		\]
		which is again $C(A,k)$.

		Finally, take any $k \in \N$ with $k \in [\tilde q_2-1, \tilde q_3-2]$.
		Since $k+1 \in [\tilde q_2, \tilde q_3)$ we have $\delta_{k+1} \geq \delta_{\tilde q_2}$.
		Also, since $F_\pm$ are increasing functions with respect to all their coordinates and $\alpha_k, \beta_k$ are decreasing sequences,
		\[
			\max\{F_+(\alpha_k, \beta_k), F_-(\alpha_k)\}
			\leq \max\{F_+(s_{\tilde q_2 -1}, t_{\tilde q_2 -1}), F_-(s_{\tilde q_2 -1})\}
			\leq \delta_{\tilde q_2} 
			\leq \delta_{k+1},
		\]
		where we used that $C(A, \tilde q_2-1)$ holds.
	\end{proof}

\subsection{Examples and description of the set $\mathbb A$}
	\label{sec_examples}

	The inequalities in Theorem \ref{thm_estimates} can be verified with a simple computer program that handles symbolic computations.
	In Figure \ref{fig_goodR} we show a sample of some values of $R_A$ for which the conditions $C(A,k)$ are verified for all $k \geq 2$, by means of Theorem \ref{thm_estimates}.

	In this section we will write the coefficients in the continued fraction of $\gamma$, as $a_j(\gamma)$.
	The corresponding area is computed from \eqref{eq_def_gamma} as
	\[A(\gamma) = \frac 43 \pi \tan(\pi \gamma)^3.\]
	
	\begin{figure}
		\caption{Some values of $R_A$ obtained using numerical simulations. The green lines represent the values of $R_A$ for which the conditions $C(A,k)$ are verified for all $k \geq 2$. The red lines represent the values of $R_A$ for which Theorem \ref{thm_estimates} failed to guarantee the conditions.
		All the computations were done in Wolfram Mathematica 14.2.}
		\label{fig_goodR}
		\includegraphics[width=\textwidth]{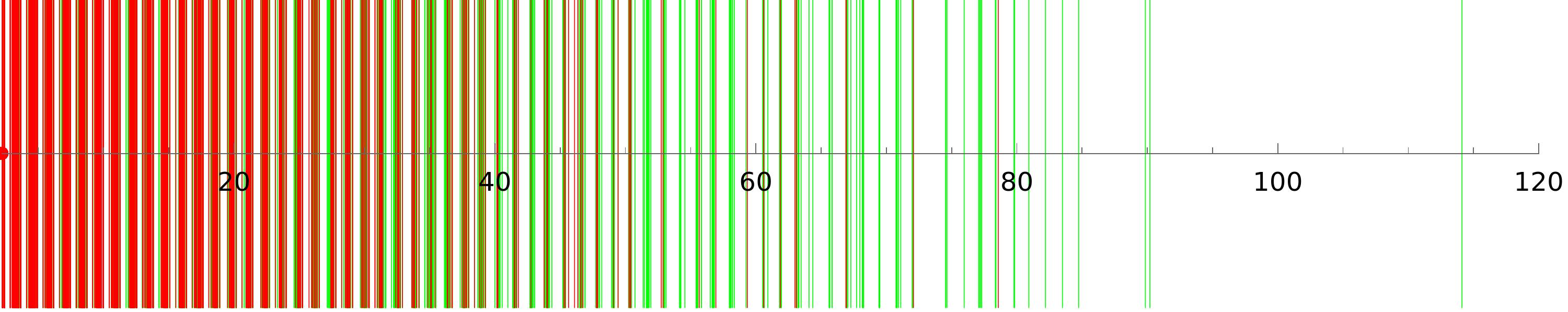}

	\end{figure}

	Let us start with a very simple example. Take
	\begin{align}
		\label{eq_RA_1}
		\gamma_0 &= \frac{1}{14} \left(\sqrt{2}+4\right) , \\
		R_0 &= \frac 1 { \cos(\pi \gamma_0) } \cong 2.87037, \\
		A_0 &= A(\gamma_0) = \frac 43 \pi \tan(\pi \gamma_0)^3 \cong 81.5849.
	\end{align}
	As a continued fraction, we have
	\[ \gamma_0 =  [0, 2, 1, 1, 2, 2, 2, \dots ], \]
	where $a_k(\gamma_0) = 2$ for $k \geq 4$.
	The first denominators of the convergents are
	\[ {q}_2=3,{q}_3=5,{q}_4=13,{q}_5=31,\]
	and the coefficients $\tilde{q}_j$ are
	\[ \tilde{q}_2=2,\tilde{q}_3=3,\tilde{q}_4=7,\tilde{q}_5=16.\]
	One verifies condition \eqref{eq_firstj0} for $j_0=4$.
	Condition \eqref{eq_conditions_1} in Theorem \ref{thm_estimates} guarantees that $C(A_0, k)$ holds for every $k \geq \tilde q_4 = 7$.
	Then one verifies by that $C(A_0,2)$ through $C(A_0,6)$ hold, so Theorem \ref{thm_main} holds true for this value of $A_0$.
	In other words, if a convex body $C$ satisfies the hypotheses of Theorem \ref{thm_main} with $A = A_0$ then $C$ is the Euclidean ball of radius $R_0$.

	We have shown that the set $\mathbb A$ is at least non-empty.
	Now let us show that $\mathbb A$ has positive Hausdorff dimension, and in particular is infinite.
	Take $\gamma_1 = \frac{1}{94} \left(\sqrt{2}+40\right)$. Its continued fraction is
	\[ \gamma_1 = [0, 2, 3, 1, 2, 2, 2, \ldots].\]
	Again define $R_1, A_1$ as in \eqref{eq_RA_1}.
	Condition \eqref{eq_firstj0} is verified for $j_0=3$, which guarantees $C(A_1, k)$ for every $k\geq \tilde q_3 = 5$ and one checks that $C(A_1,2), C(A_1, 3), C(A_1,4)$ hold.
	Moreover, \eqref{eq_firstj0} still holds if we replace the term $\max_{j \geq j_0+2}\{a_j\}$ by $7$.
	Let 
	\[E[a_0, \ldots, a_N] = \{[a_0, \ldots, a_N, x]: x>0\}\]
	be the set of positive real numbers whose continued fraction starts with $a_0, \ldots, a_N$.
	Let $\mathcal E_{\gamma_1, N}$ be the set of $x \in E[a_0(\gamma_1), \ldots, a_N(\gamma_1)]$ such that $a_j(x)$ belongs to the set $\{2,4,6\}$ for $j>N$.
	By the right inequality of Proposition \ref{prop_estimate}, any two numbers whose continued fraction coincide at the first $N$ terms, must be at distance at most $\frac 2{q_N q_{N+1}}$.
	Since $C(A_1,2), C(A_1, 3), C(A_1,4)$ are strict inequalities, we can take $N$ sufficiently large so that  $C(A(\gamma),2), C(A(\gamma), 3), C(A(\gamma),4)$ still hold for every $\gamma \in \mathcal E_{\gamma_1, N}$.
	In this case we have $A(\gamma) \in \mathbb A$ for all $\gamma \in \mathcal E_{\gamma_1, N}$, which is an uncountable set.
	Moreover, by a classical result of I. J. Good \cite[Theorem 11]{good1941fractional}, together with the elementary fact \cite[Lemma 1]{good1941fractional}, the Hausdorff dimension of $\mathcal E_{\gamma_1, N}$  is strictly positive. By the invariance of the Hausdorff dimension through diffeomorphisms, the same happens with $\{A(\gamma) : \gamma \in \mathcal E_{\gamma_1, N} \} \subseteq \mathbb A$.
	We conclude that the Hausdorff dimension of $\mathbb A$ is strictly positive.

	Let us show a family of numbers $\gamma$ which is considerably larger:
	We stay with the numbers $\gamma$ close to $\gamma_1$, but this time we impose a polynomial growth on $a_j(\gamma_1)$, instead of constant.
	We assume that $\gamma \in E[a_0(\gamma_1), \ldots, a_N(\gamma_1)]$ and that $a_j \leq j^2$ for $j \geq 4$.
	Equation \eqref{eq_conditions_1} holds for $j \geq j_0 = 4$, this is,
	\begin{equation}
		\label{eq_firstj1}
		\frac{\pi(a_{j+2}+2)}{\cos(\pi / q_{j}) \sqrt{R^2 \cos(\pi / q_{j+1})^2 - 1}}  
		\leq \frac{\pi((j+2)^2+2)}{\cos(\pi / q_{j}) \sqrt{R^2 \cos(\pi / q_{j+1})^2 - 1}} 
		\leq q_j.
	\end{equation}
	To see the last inequality, one checks that
	\begin{equation}
		\label{eq_comparison_with_exponential}
		\frac{\pi}{\cos(\pi / q_{j}) \sqrt{R^2 \cos(\pi / q_{j+1})^2 - 1}} \leq \frac{2^{\frac{j-1}2}}{(j+2)^2 + 2}
	\end{equation}
	holds for $j = 17$, and that the left-hand side is decreasing, while the right-hand side is increasing for $j\geq 17$. Then \eqref{eq_comparison_with_exponential} holds for every $j \geq 17$ and the remaining cases $j=4, \ldots, 16$ are checked separately.

	Let $Q_n = \{x>0:a_j(x) \leq j^2, \text{ for } j \geq n+1\}$.
	As before, we see that there is a large value of $N$, such that any irrational $\gamma$ in $Q_N \cap E[a_0(\gamma_1), \ldots, a_N(\gamma_1)]$ with even coefficients $a_j(\gamma_1)$ for $j \geq 4$, belongs to $\mathbb A$.

	We argue that this set is large in some sense.
	It can be shown that $Q_N \cap E[a_0(\gamma), \ldots, a_N(\gamma)]$ actually has positive measure.
	Indeed, by \cite[Theorem 30]{khinchin1964continued}, the increasing union of the $Q_n$ for $n \geq 4$ is a full-measure set.
	Unfortunately, the condition that $a_j(\gamma_1)$ is even for $j \geq 4$ forces this set of points to have measure zero.

	\section{Concluding Remarks}

	Finally, let us observe that the condition imposed on the parity of $a_j(\gamma)$ for $j \geq 4$ can be dropped.
	The first denominator is $q_0 = 1$, and formula \eqref{eq_recursion_q} shows that if $q_{j-2}$ is odd and $q_{j-1}$ is even, then $q_j$ must be again odd.
	This way, we have an infinite sequence of odd denominators, which can be used to obtain bounds as in Theorem \ref{thm_estimates}.
	Using only this sequence instead of the full sequence $q_j$, the conditions in Theorem \ref{thm_estimates} become considerably more complicated, and the bounds practically intractable, where millions of inequalities must be verified.
	Then we decide to leave the following question open:
	\begin{question}
		Is $\mathbb A$ a set with positive Lebesgue measure?
	\end{question}


\bibliographystyle{abbrv}

\bibliography{../references}

\end{document}